\documentclass[12pt]{amsart}
\usepackage[cp1250]{inputenc}
\usepackage{graphicx, lpic}
\usepackage{amssymb, amsmath}
\usepackage[sc]{mathpazo}

\theoremstyle{plain}
\newtheorem{theorem}{Theorem}
\newtheorem{proposition}[theorem]{Proposition}
\newtheorem{lemma}[theorem]{Lemma}
\newtheorem{corollary}[theorem]{Corollary}

\theoremstyle{definition}
\newtheorem{definition}[theorem]{Definition}

\newtheorem{remark}[theorem]{Remark}

\allowdisplaybreaks

\begin{document}

\title[Yoshikawa eighth move and a minimal set of band moves]{Independence of Yoshikawa eighth move and\\ a minimal generating set of band moves}

\author[M. Jab{\l}onowski]{Micha{\l} Jab{\l}onowski}

\dedicatory{}

\address{Institute of Mathematics, Faculty of Mathematics, Physics and Informatics, University of Gda\'nsk, 80-308 Gda\'nsk, Poland}

\email{michal.jablonowski@gmail.com}



\date{\today}

\begin{abstract}

Yoshikawa moves were introduced at least quarter-century ago and are still actively used by researchers. For any marked graph diagram we will define its twisted diagram and its mirror cut surface. By using a surface-link group of a mirror cut surface of a twisted diagram we will prove the independence of Yoshikawa eighth move. As a consequence we will establish a minimal generating set of band moves for links with bands.

\end{abstract}

\maketitle

\section{Introduction}\label{s1}

A surface-link is a closed $2$-manifold smoothly (or piecewise linearly and locally flatly) embedded in the Euclidean $4$-space. A marked graph diagram is a link diagram possibly with $4$-valent vertices equipped with markers that can represent a surface-link.
\par
It is known that the set of ten types of moves $\{\Omega_1, \ldots, \Omega_8, \Omega_4', \Omega_6'\}$ (presented in Fig.\;\ref{pic002}), called Yoshikawa moves, is a generating set of moves that relates two marked graph diagrams presenting equivalent surface-links. In (\cite{JKL13}, \cite{JKL15}) it is shown that any Yoshikawa move from the set $\{\Omega_1, \Omega_2, \Omega_3, \Omega_6, \Omega_6', \Omega_7\}$ is independent from the other nine types. It is an open problem (see \cite{JKL15}) whether the move $\Omega_8$ is independent from the other Yoshikawa moves.
\par
We can translate the marked graph diagrams in $\mathbb{R}^2$ to links with bands in $\mathbb{R}^3$, then instead of ten types of moves we have the set of four generating types of moves $\{M_1, M_2, M_3, M_4\}$ (presented in Fig.\;\ref{MJ_102}). Therefore, independence of some Yoshikawa moves leads us to obtain a minimal generating set of moves on links with bands. We will prove the following two main theorems.

\begin{theorem}\label{thm1}

The Yoshikawa move $\Omega_8$ cannot be realized by a finite sequence of Yoshikawa moves of the other nine types.

\end{theorem}

\begin{theorem}\label{thm2}

The set $\{M_1, M_2, M_3, M_4\}$ is a minimal generating set of moves on links with bands.

\end{theorem}

This paper is organized as follows. In Sec.\;\ref{s2}, we review marked graphs corresponding to surface-links in the hyperbolic splitting position and their diagrams. In Sec.\;\ref{s3} for any marked graph diagram we will define its twisted diagram and its mirror cut surface. By using a surface-link group of a mirror cut surface of a twisted diagram we will prove Thm.\;\ref{thm1}. In Sec.\;\ref{s4} we will prove Thm.\;\ref{thm2}.

\section{Preliminaries}\label{s2}

Throughout this paper, we work in the standard smooth category. An embedding (or its image) of a closed (i.e. compact, without boundary) surface $F$ into $\mathbb{R}^4$ is called a \emph{surface-link} (or \emph{surface-knot} if it is connected). Two surface-links are \emph{equivalent} (or have the same \emph{type}, which is denoted also by $\cong$) if there exists an orientation preserving homeomorphism of the four-space $\mathbb{R}^4$ to itself (or equivalently auto-homeomorphism of the four-sphere $\mathbb{S}^4$), mapping one of those surfaces onto the other.
\par
We will use a word \emph{classical}, thinking about theory of embeddings of circles $S^1\sqcup\ldots\sqcup S^1\hookrightarrow \mathbb{R}^3$ modulo ambient isotopy in $\mathbb{R}^3$ with their planar or spherical generic projections.
\par
To describe a knotted surface in $\mathbb{R}^4$, we will use transverse cross-sections $\mathbb{R}^3\times\{t\}\subset\mathbb{R}^4$ for $t\in\mathbb{R}$, denoted by $\mathbb{R}^3_t$. This method introduced by Fox and Milnor was presented in \cite{Fox62}.
\par
It is well-known (\cite{Lom81}, \cite{KSS82}, \cite{Kam89}) that any surface-link $F$ admits a \emph{hyperbolic splitting}, i.e. there exists a surface-link $F'$ satisfying the following: $F'$ is equivalent to $F$ and has only finitely many Morse's critical points, all maximal points of $F'$ lie in $\mathbb{R}^3_1$, all minimal points of $F'$ lie in $\mathbb{R}^3_{-1}$, all saddle points of $F'$ lie in $\mathbb{R}^3_0$.

\begin{figure}[ht]
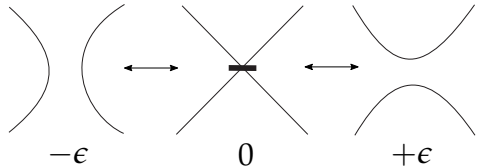

\begin{center}
\begin{lpic}[b(0.5cm)]{./pictures/m001(6.2cm)}
	\lbl[t]{15,-2;$-\epsilon$}
  \lbl[t]{60,-2;$0$}
  \lbl[t]{102,-2;$+\epsilon$}
	\end{lpic}
		\caption{Rules for smoothing a marker.\label{pic001}}
\end{center}
\end{figure}

The zero section $\mathbb{R}^3_0\cap F'$ of the surface $F'$ in the hyperbolic splitting described above gives us then a $4$-regular graph. We assign to each vertex a \emph{marker} that informs us about one of the two possible types of saddle points (see Fig.\;\ref{pic001}) depending on the shape of the section $\mathbb{R}^3_{-\epsilon}\cap F'$ or $\mathbb{R}^3_{\epsilon}\cap F'$ for a small real number $\epsilon>0$. The resulting (rigid-vertex) graph is called a \emph{marked graph} presenting $F$.
\par
Making a projection in general position of this graph to $\mathbb{R}^2\times\{0\}\times\{0\}\subset\mathbb{R}^4$ and assigning types of classical crossings between regular arcs, we obtain a \emph{marked graph diagram} . For a marked graph diagram $D$, we denote by $L_+(D)$ and $L_-(D)$ the classical link diagrams obtained from $D$ by smoothing every vertex as presented in Fig.\;\ref{pic001} for $+\epsilon$ and $-\epsilon$ case respectively. We call $L_+(D)$ and $L_-(D)$ the \emph{positive resolution} and the \emph{negative resolution} of $D$, respectively.
\par
Any abstractly created marked graph diagram is a \emph{ch-diagram} (or it is \emph{admissible}) if and only if both its resolutions are trivial classical link diagrams. 
\par
A \emph{band} on a link $L$ is an image of an embedding $b:I\times I\to\mathbb{R}^3$ intersecting the link $L$ precisely in the subset $b(\partial I\times I)$. A \emph{link with bands} $LB$ in $\mathbb R^3$ is a pair $(L, B)$ consisting of a link $L$ in $\mathbb R^3$ and a finite set $B=\{b_1, \dots, b_n\}$ of pairwise disjoint $n$ bands spanning $L$.

\begin{figure}[ht]
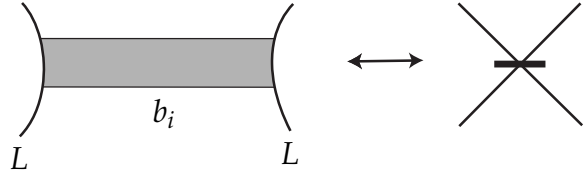

\begin{center}
\begin{lpic}[b(0.5cm)]{./pictures/MJ_100(7.5cm)}
  \lbl[t]{30,8;$b_i$}
   \lbl[t]{0,-2;$L$}
   \lbl[t]{56,-1;$L$}
	\end{lpic}
	\caption{A band corresponding to a marked vertex.\label{MJ_100}}
\end{center}
\end{figure}

By an ambient isotopy of $\mathbb R^3$, we shorten the bands of a link with bands $LB$ so that each band is contained in a small $2$-disk. Replacing the neighborhood of each band to the neighborhood of a marked vertex as in Fig.\;\ref{MJ_100}, we obtain a marked graph, called a \emph{marked graph associated with} $LB$. Conversely, when a marked graph $G$ in $\mathbb R^3$ is given, by replacing each marked vertex with a band as in Fig.\;\ref{MJ_100}, we obtain a link with bands $LB(G)$, called a \emph{link with bands associated with} $G$. 
\par
Let $D$ be a ch-diagram with associated link with bands $LB(D)=(L,B)$, $L=L_-(D)$, $B=\{b_1, \dots, b_n\}$ and $\Delta_1,\ldots,\Delta_a \subset \mathbb{R}^3$ be mutually disjoint $2$-disks with $\partial(\cup_{j=1}^a\Delta_j)= L_+(D)$, and let $\Delta_1',\ldots,\Delta_b' \subset \mathbb{R}^3$ be mutually disjoint $2$-disks with $\partial(\cup_{k=1}^b\Delta_k')= L_-(D)$. We define $\mathcal S(D) \subset \mathbb{R}^3 \times \mathbb{R} = \mathbb{R}^4 $ a \emph{surface-link corresponding to diagram} $D$ by the following cross-sections.

$$
(\mathbb{R}^3_t, \mathcal{S}(D) \cap \mathbb{R}^3_t)=\left\{%
\begin{array}{ll}
(\mathbb R^3, \emptyset) & \hbox{for $t > 1$,}\\
(\mathbb R^3, L_+(D) \cup (\cup_{j=1}^a\Delta_j)) & \hbox{for $t = 1$,} \\
(\mathbb R^3, L_+(D)) & \hbox{for $0 < t < 1$,} \\
(\mathbb R^3, L_-(D) \cup (\cup_{i=1}^n b_i)) & \hbox{for $t = 0$,} \\
(\mathbb R^3, L_-(D)) & \hbox{for $-1 < t < 0$,} \\
(\mathbb R^3, L_-(D) \cup (\cup_{k=1}^b\Delta_k')) & \hbox{for $t = -1$,} \\
(\mathbb R^3, \emptyset) & \hbox{for $ t < -1$.} \\
\end{array}
\right.
$$

It is known that the surface-knot type of $S(D)$ does not depend on choices of trivial disks (cf. \cite{KSS82}). It is straightforward from the construction of $S(D)$ that $D$ is a marked graph diagram presenting $S(D)$.
\par
In \cite{Yos94} Yoshikawa introduced local moves on admissible marked graph diagrams that do not change corresponding surface-link types and conjectured that the converse is also true. It was resolved (\cite{Swe01}, \cite{KeaKur08}, \cite{JKL15}) as follows. Any two marked graph diagrams representing the same type of surface-link are related by a finite sequence of Yoshikawa local moves presented in Fig.\;\ref{pic002} and an isotopy of the diagram in $\mathbb{R}^2$.

\begin{figure}[ht]
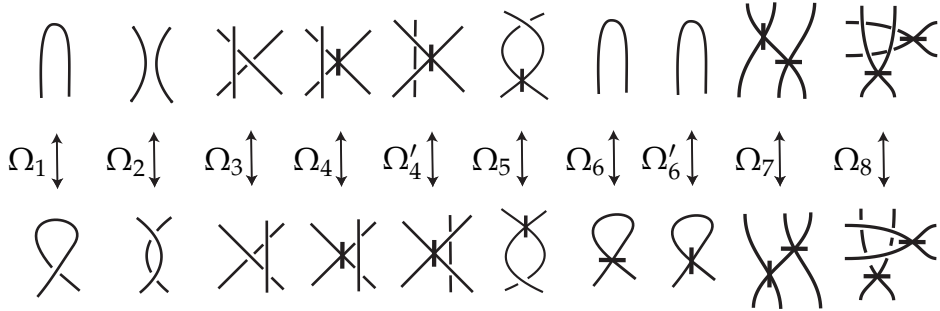

\begin{center}
\begin{lpic}[l(0.4cm)]{./pictures/m002(12cm)}
	\lbl[r]{3,33;$\Omega_1$}
  \lbl[r]{25,33;$\Omega_2$}
	\lbl[r]{47,33;$\Omega_3$}
	\lbl[r]{67,33;$\Omega_4$}
	\lbl[r]{87,33;$\Omega_4'$}
	\lbl[r]{107,33;$\Omega_5$}
	\lbl[r]{128,33;$\Omega_6$}
	\lbl[r]{145,33;$\Omega_6'$}
	\lbl[r]{166,33;$\Omega_7$}
	\lbl[r]{188,33;$\Omega_8$}
	\end{lpic}
\caption{A set of Yoshikawa moves.\label{pic002}}
\end{center}
\end{figure}

An orientable surface-link in $\mathbb{R}^4$ is \emph{unknotted} if it is equivalent to a surface embedded in $\mathbb{R}^3\times\{0\}\subset\mathbb{R}^4$. A marked graph diagram for an unknotted \emph{standard sphere} $\mathbb{S}^2$ is shown in Fig.\;\ref{m005}(a), an unknotted \emph{standard torus} $\mathbb{T}^2$ is in Fig.\;\ref{m005}(d). 
\par
An embedded projective plane $\mathbb{P}^2$ in $\mathbb{R}^4$ is \emph{unknotted} if it is equivalent to a surface whose marked graph diagram is of an unknotted \emph{standard projective plane}, which depicted in Fig.\;\ref{m005}(b) (resp. Fig.\;\ref{m005}(c)) is a \emph{positive} $\mathbb{P}^2_+$ (resp. a \emph{negative} $\mathbb{P}^2_-$). A non-orientable surface-knot is \emph{unknotted} if it is equivalent to some finite connected sum of unknotted projective planes.

\begin{figure}[ht]
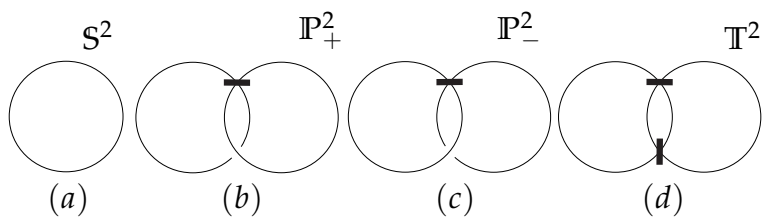

\begin{center}
\begin{lpic}[b(0.7cm),t(0.7cm)]{./pictures/m005(10cm)}
  \lbl[b]{18,25;$\mathbb{S}^2$}
  \lbl[b]{63,25;$\mathbb{P}^2_+$}
	\lbl[b]{103,25;$\mathbb{P}^2_-$}
  \lbl[b]{148,25;$\mathbb{T}^2$}
  \lbl[t]{12,-2;$(a)$}
  \lbl[t]{47,-2;$(b)$}
	\lbl[t]{90,-2;$(c)$}
  \lbl[t]{132,-2;$(d)$}
	\end{lpic}
\caption{Examples of the unknotted surfaces.\label{m005}}
\end{center}
\end{figure}

\begin{proposition}[\cite{Yos94}]\label{prop1}

The relations $H1, H2, H3, H4$ presented in Fig.\;\ref{r4} applied on ch-diagrams do not change the type of associated surface-link.

\end{proposition}

\begin{figure}[ht]
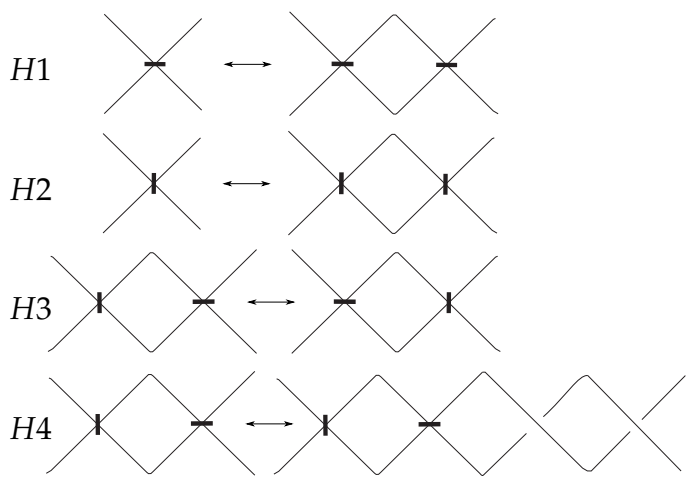

\begin{center}
\begin{lpic}[l(0.9cm)]{./pictures/r4(8.5cm)}
\lbl[r]{2,136;$H1$}
\lbl[r]{2,96;$H2$}
\lbl[r]{2,56;$H3$}
\lbl[r]{2,16;$H4$}
	\end{lpic}
\caption{Additional Yoshikawa relations.\label{r4}}
\end{center}
\end{figure}

\section{Independence of Yoshikawa eighth move}\label{s3}

A \emph{standard stabilization/destabilization} is defined by adding/deleting an unknotted handle attached locally. The \emph{surface-link group} of a surface-link $F$, denoted by $Gr(F)$, is the fundamental group $\pi_1(\mathbb{R}^4\backslash \text{int}(N(F)))$ where $N(F)$ is a tubular neighborhood of $F$.

\begin{definition}
Let $D$ be a marked graph diagram such that $L_-(D)$ is a trivial link and let $\Delta_1',\ldots,\Delta_b' \subset \mathbb{R}^3$ be mutually disjoint $2$-disks with $\partial(\cup_{k=1}^b\Delta_k')=L_-(D)$. Denote by $LB(D)$ the associated link with bounds $(L, B)$ with $L=L_-(D)$ and $B=\{b_1, \dots, b_n\}$, and let $B^*=\{b^*_1, \dots, b^*_n\}$ be the dual bands of $B$ (associated to changing corresponding markers to their second type). We define a \emph{mirror cut surface} of $D$ denoted by $MC(D)\subset\mathbb{R}^3\times\mathbb{R}=\mathbb{R}^4$ as a surface-link defined by the following cross-sections.

$$
(\mathbb{R}^3_t, MC(D) \cap \mathbb{R}^3_t)=\left\{%
\begin{array}{ll}
(\mathbb R^3, \emptyset) & \hbox{for $t > 2$,}\\
(\mathbb R^3, L_-(D) \cup (\cup_{k=1}^b\Delta_k')) & \hbox{for $t = 2$,} \\
(\mathbb R^3, L_-(D)) & \hbox{for $1 < t < 2$,} \\
(\mathbb R^3, L_+(D) \cup (\cup_{i=1}^n b^*_i)) & \hbox{for $t = 1$,} \\
(\mathbb R^3, L_+(D)) & \hbox{for $0 < t < 1$,} \\
(\mathbb R^3, L_-(D) \cup (\cup_{i=1}^n b_i)) & \hbox{for $t = 0$,} \\
(\mathbb R^3, L_-(D)) & \hbox{for $-1 < t < 0$,} \\
(\mathbb R^3, L_-(D) \cup (\cup_{k=1}^b\Delta_k')) & \hbox{for $t = -1$,} \\
(\mathbb R^3, \emptyset) & \hbox{for $ t < -1$.} \\
\end{array}
\right.
$$

\end{definition}

The mirror cut surface of a given diagram $D$ can be intuitively seen as the result after making cut of the surface-link $\mathcal S(D)$ corresponding to $D$ in the level of $t=\frac{1}{2}$, making copy of its lower part and gluing it with the mirror image of its copy by their common boundary link $L_+(D)$. In other words put a mirror in the cut level and see the whole surface-link from below. A depiction of saddles in the middle cross-sections is presented in Fig.\;\ref{r1}.

\begin{figure}[ht]
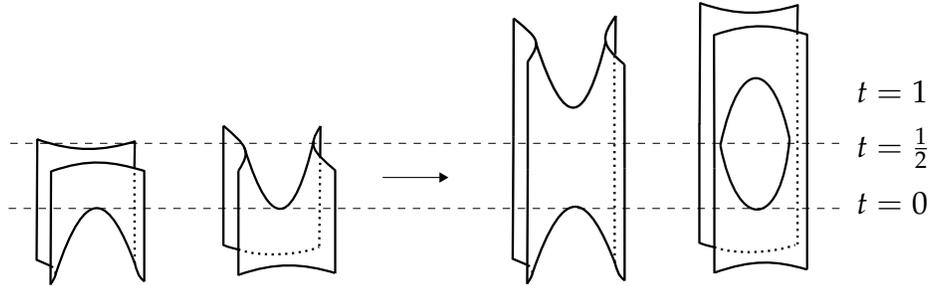

\begin{center}
\begin{lpic}[r(1.3cm)]{./pictures/r1(11.1cm)}
\lbl[l]{212,49;$t=1$}
\lbl[l]{212,35;$t=\frac{1}{2}$}
\lbl[l]{212,21;$t=0$}
	\end{lpic}
\caption{Saddles change by the mirror cut transformation.\label{r1}}
\end{center}
\end{figure}

\begin{definition}

For a marked graph diagram $D$, let $\mathcal{O}(D)$ denote the set of all possible orientations of $L_-(D)$. For a fixed $o\in\mathcal{O}(D)$, define \emph{twisted graph diagram} of $D$ denoted by $D^t$ as follows. If in a neighborhood of a marked vertex of $D$, the orientation of strands are coherent (see the top of Fig.\;\ref{r2}), call this vertex \emph{type} $d$ and do not change diagram $D^t$ in the neighborhood. If the corresponding orientation near a marked vertex in $D$ is non-coherent, see the left-bottom of Fig.\;\ref{r2}, then call this vertex \emph{type} $e$ and change it by adding a classical crossing as shown in the right-bottom of the figure. Note that the marked graph diagram $D^t$ may not be admissible.

\end{definition}

\begin{figure}[ht]
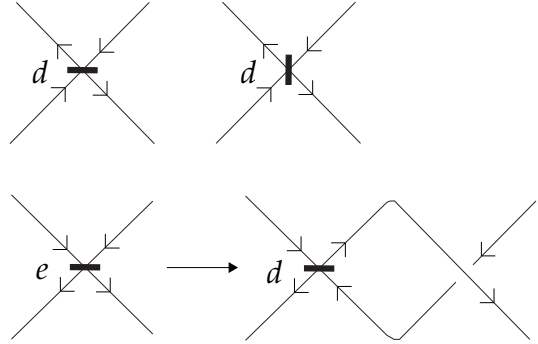

\begin{center}
\begin{lpic}[]{./pictures/r2(7cm)}
\lbl[r]{9,62;$d$}
\lbl[r]{57,62;$d$}
\lbl[r]{9,16;$e$}
\lbl[r]{63,16;$d$}
	\end{lpic}
\caption{Rules for twisting a marked graph diagram.\label{r2}}
\end{center}
\end{figure}

By ambient isotopy in $\mathbb{R}^4$ we can obtain the following ch-diagram of the hyperbolic splitting of the mirror cut surface for the twisted ch-diagram.

\begin{proposition}\label{prop2}

Let $D$ be a ch-diagram with a fixed $o\in\mathcal{O}(D)$. A marked graph diagram presenting $MC(D^t)$, denoted by $\mathcal{W}(D)$ can be obtained by replacing each marked vertex from $D$ by the rule presenting in Fig.\;\ref{r3}, with the orientation in the neighborhood of the new vertex induced from $o$. 

\end{proposition}

\begin{figure}[ht]
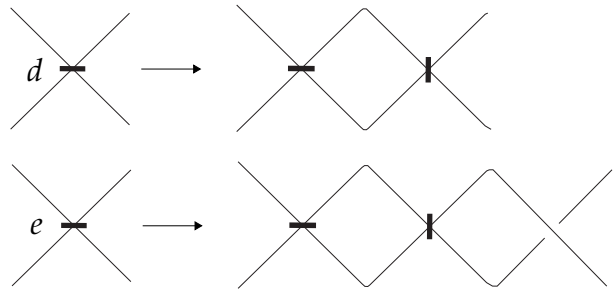

\begin{center}
\begin{lpic}[l(0.7cm)]{./pictures/r3(8cm)}
\lbl[r]{9,60;$d$}
\lbl[r]{9,17;$e$}
	\end{lpic}
\caption{A marked graph diagram presentation of $\mathcal{W}(D)$.\label{r3}}
\end{center}
\end{figure}

\begin{remark}

Note that the moves from Prop.\;\ref{prop1} can be obtained by a combination of only Yoshikawa types of moves $\Omega_1, \ldots, \Omega_7, \Omega_4', \Omega_6'$ and ambient isotopy in $\mathbb{R}^2$. If we change the order of marked crossings and classical crossings appeared in the bottom-right picture in Fig.\;\ref{r3}, the local diagram can be transformed into the shape shown in the bottom-right picture in Fig.\;\ref{r3} by using $H3, H4, \Omega_2, \Omega_5$.

\end{remark}

\begin{lemma}\label{lem2}

Let $V$ be a set of marked vertices of a marked graph diagram $D_1$ presenting a surface-link $F_1$. Define $D_2$ as a graph diagram obtained by switching markers on each vertex from $V$ to its second type (i.e. rotating them by $90^{\circ}$ around vertex). If $D_2$ is a ch-diagram (i.e is admissible) presenting surface-link $F_2$ then $Gr(F_2)\cong Gr(F_1)$.

\end{lemma}

\begin{proof}

By the group calculation method presented in \cite{Fox62} we see that the presentation relations remains the same when markers type are switched and the corresponding surface is still closed. So the surface-link group (obtained from a marked graph diagram) cannot detect types of markers, as long as the diagram remains admissible.

\end{proof}

\begin{remark}

As a consequence of Lem.\;\ref{lem2} and Prop.\;\ref{prop1}, and by looking at Yoshikawa table of surface-link marked graph diagrams \cite{Yos94} we can simply obtain (without calculations) the following isomorphisms $Gr(2_1^1)\cong Gr(0_1)$ and $Gr(10_1^{-2,-2})\cong Gr(8_1^{-1,-1})$.

\end{remark}

Considering a marked graph diagram of the unknotted standard torus presented in Fig.\;\ref{m005}(d), this together with Lem.\;\ref{lem2} and Prop.\;\ref{prop1} yields the following.

\begin{corollary}\label{cor1}

A standard stabilization/destabilization of a surface-link does not change its group.

\end{corollary}

For a ch-diagram $D$ let us define a map $\mathcal{M}\left(D\right)$ as a (unordered) set over all possible orientations
$$\mathcal{M}\left(D\right)=\{Gr(MC(D^t))\;|\;o\in\mathcal{O}(D)\}.$$

\begin{proposition}\label{prop3}

For a marked graph diagram $D$ the set $\mathcal{M}\left(D\right)$ is an invariant of nine Yoshikawa type of moves $\Omega_1, \ldots, \Omega_7, \Omega_4', \Omega_6'$ .

\end{proposition}

\begin{proof}

For a fixed $o\in\mathcal{O}(D)$ we investigate how this nine Yoshikawa moves performed on a diagram $D$ changes the marked graph diagram $\mathcal{W}(D)$ obtained as in Prop.\;\ref{prop2}. It is clear that moves $\Omega_1, \Omega_2, \Omega_3$ do not change the type of surface-link $S(\mathcal{W}(D))\cong MC(D^t)$ because this move do not involve marked vertices. Performing moves $\Omega_4, \Omega_4'$ on $D$ produces changes on $\mathcal{W}(D)$ that can be easily obtained using moves $\Omega_1, \ldots, \Omega_4, \Omega_4'$, no matter what the type $d$ or $e$ the involved marked vertex was.
\par
We see that in the case of $\Omega_5$ move the diagram $\mathcal{W}(D)$ is changed by moves that can be generated by the moves $\Omega_1, \ldots, \Omega_5, \Omega_4'$, because in this case the orientation type ($d$ or $e$) of involved marked vertex does not change, and this transformation can be done by an ambient isotopy in $\mathbb{R}^3$.
\par
Let us proceed to moves $\Omega_6, \Omega_6'$ on the diagram $D$, the corresponding changes on $\mathcal{W}(D)$ are shown in Fig.\;\ref{r5}. We see that in the case of $\Omega_6$ move the surface-link $S(\mathcal{W}(D))$ is changed by a standard stabilization/destabilization, because the marked vertex involved has to be of type $d$, so by Cor.\;\ref{cor1} the group $Gr(S(\mathcal{W}(D)))$ does not change its isomorphism class. In the case of $\Omega_6'$ move the diagram $\mathcal{W}(D)$ is changed by moves that can be generated by the moves $\Omega_1, \Omega_6, \Omega_6'$.

\begin{figure}[ht]
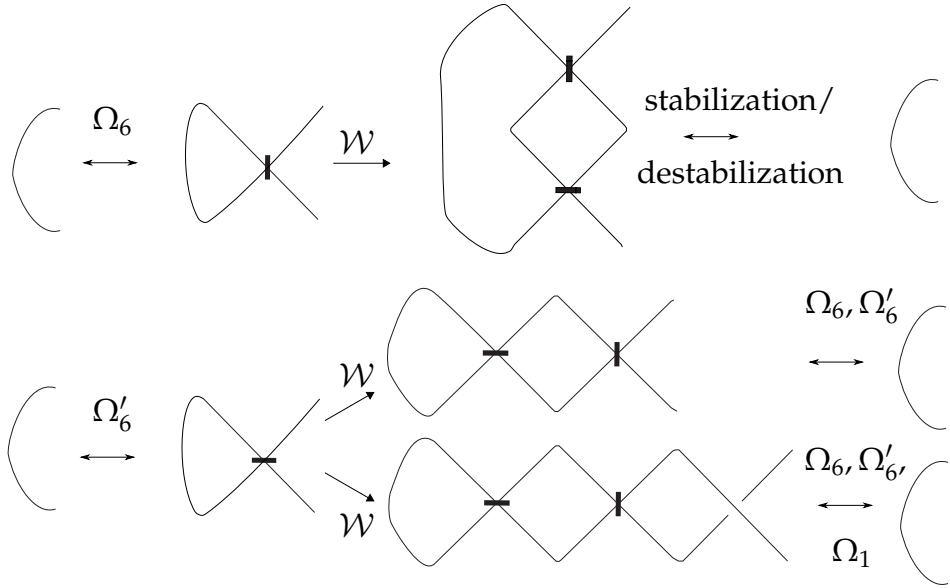

\begin{center}
\begin{lpic}[]{./pictures/r5(12.5cm)}
\lbl[b]{100,124;$\mathcal{W}$}
\lbl[b]{100,57;$\mathcal{W}$}
\lbl[b]{100,14;$\mathcal{W}$}
\lbl[b]{30,129;$\Omega_6$}
\lbl[b]{30,44;$\Omega_6'$}
\lbl[b]{210,135;stabilization/}
\lbl[b]{210,115;destabilization}
\lbl[b]{242,75;$\Omega_6, \Omega_6'$}
\lbl[b]{243,30;$\Omega_6, \Omega_6', $}
\lbl[b]{242,5;$\Omega_1$}
	\end{lpic}
\caption{Changes by $\Omega_6$ and $\Omega_6'$ moves.\label{r5}}
\end{center}
\end{figure}

Finally, we investigate the change on $\mathcal{W}(D)$ by performing $\Omega_7$ move on the diagram $D$ which is depicted in Fig.\;\ref{r6}, where (taking additionally their mirror images) all the possible cases are covered. We see there that transformations (two-sided arrows) follow from moves $H3, H4, \Omega_1, \ldots \Omega_5, \Omega_4', \Omega_7$. The case where all four marked vertex involved in move $\Omega_7$ are of type $e$ cannot occur as a consequence of definition of type $e$ and investigation of all possible cases of types of orientations of classical link $L_-(D)$ involved in the move.

\begin{figure}[ht]
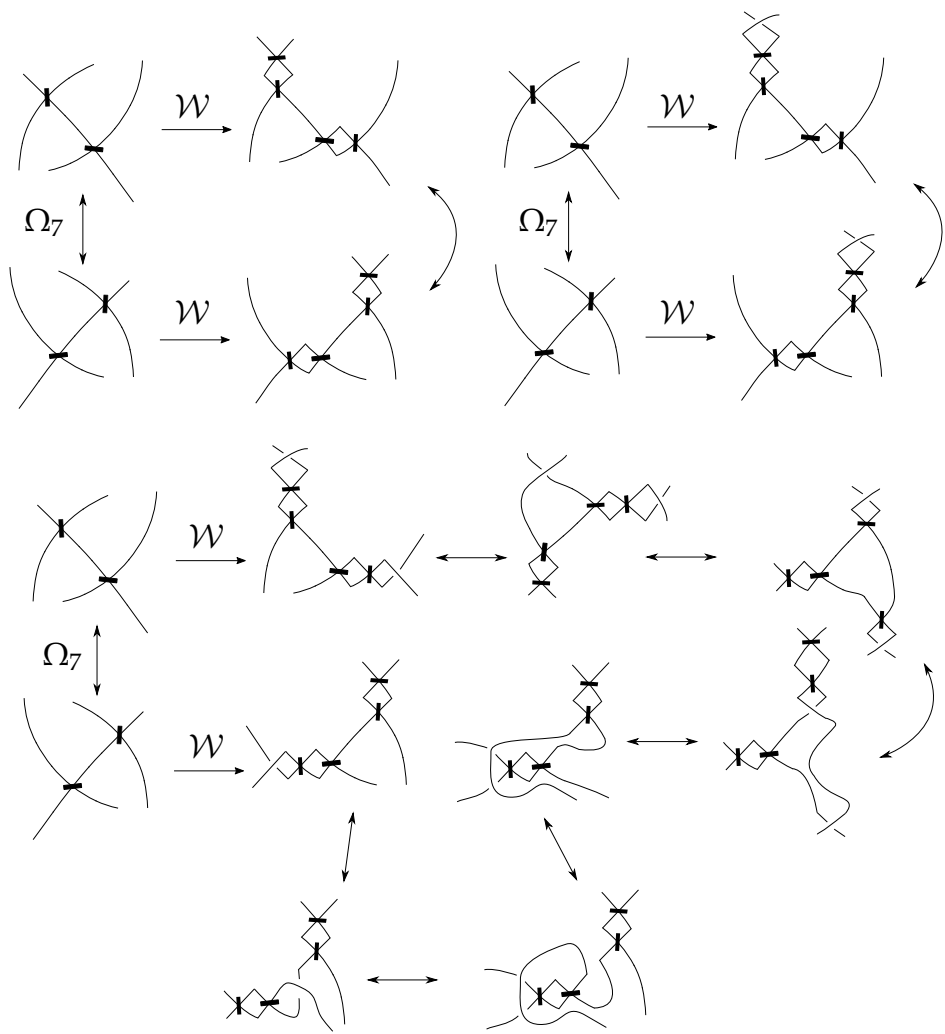

\begin{center}
\begin{lpic}[]{./pictures/r6(12.4cm)}
\lbl[r]{12,177;$\Omega_7$}
\lbl[r]{16,82;$\Omega_7$}
\lbl[r]{120,177;$\Omega_7$}
\lbl[b]{40,200;$\mathcal{W}$}
\lbl[b]{40,155;$\mathcal{W}$}
\lbl[b]{43,106;$\mathcal{W}$}
\lbl[b]{43,60;$\mathcal{W}$}
\lbl[b]{146,200;$\mathcal{W}$}
\lbl[b]{146,155;$\mathcal{W}$}
	\end{lpic}
\caption{Changes by $\Omega_7$ move.\label{r6}}
\end{center}
\end{figure}

\end{proof}

\begin{proof}[Proof of Thm.\;\ref{thm1}.]

Let $D_1$ and $D_2$ be two ch-diagrams presented in Fig.\;\ref{r7}. They are diagrams of equivalent surface-links of $\mathbb{S}^2\sqcup\mathbb{P}^2_-$ as they differ by just one Yoshikawa move (four of classical crossings are changed). However, $\mathcal{M}\left(D_1\right)\not=\mathcal{M}\left(D_2\right)$, because the set of the multivariable first elementary ideals for the corresponding elements in $\mathcal{M}\left(D_1\right)$ is $\{\left<0\right>\}$ and for $\mathcal{M}\left(D_2\right)$ is $\{\left<(x+1)(x-1),(x+1)(y-1)\right>\}$. Therefore, by Prop.\;\ref{prop3} diagram $D_1$ cannot be transformed into $D_2$ by using nine Yoshikawa types of moves $\Omega_1, \ldots, \Omega_7, \Omega_4', \Omega_6'$ and ambient isotopy in $\mathbb{R}^2$.

\begin{figure}[ht]
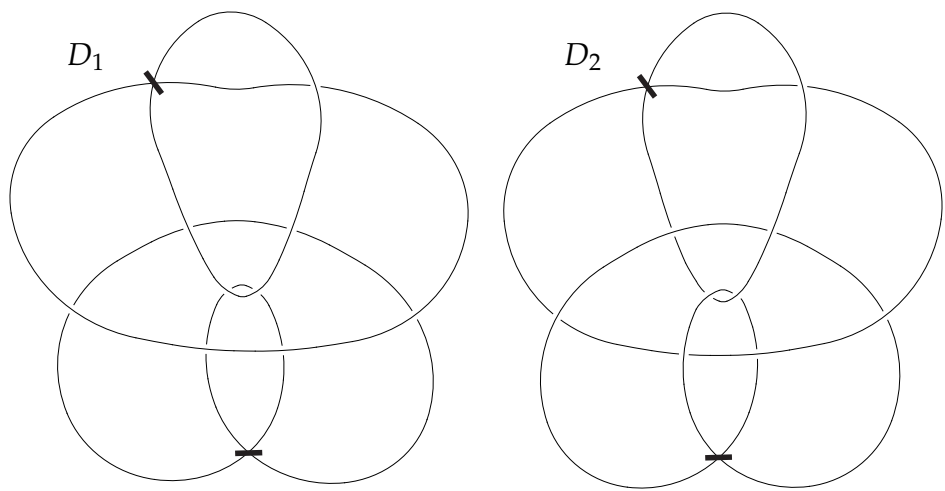

\begin{center}
\begin{lpic}[]{./pictures/r7(12.4cm)}
 \lbl[r]{25,113;$D_1$}
  \lbl[r]{155,113;$D_2$}
	\end{lpic}
\caption{Diagrams $D_1$ and $D_2$.\label{r7}}
\end{center}
\end{figure}

\end{proof}

\section{A minimal generating set of band moves}\label{s4}

Examples of the links with bands of the unknotted surfaces are presented in Fig.\;\ref{MJ_106}, for the examples of nontrivial links with bands refer to \cite{Jab16}.

\begin{figure}[ht]
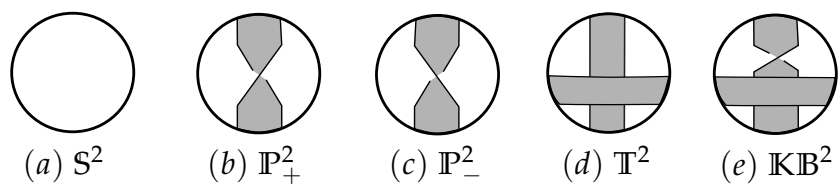

\begin{center}
\begin{lpic}[b(0.75cm)]{./pictures/MJ_106(11cm)}
  \lbl[t]{10,-2;$(a)\;\mathbb{S}^2$}
  \lbl[t]{46,-2;$(b)\;\mathbb{P}^2_+$}
	\lbl[t]{80,-2;$(c)\;\mathbb{P}^2_-$}
  \lbl[t]{112,-2;$(d)\;\mathbb{T}^2$}
	\lbl[t]{144,-2;$(e)\;\mathbb{KB}^2$}
	\end{lpic}
		\caption{Examples of the unknotted surfaces.\label{MJ_106}}
		\end{center}
\end{figure}

We have the analogous (to the marked graph diagram case) moves for links with bands (\cite{Swe01}, \cite{KeaKur08}), i.e. any two links with bands representing the same type of surface-link are related by a finite sequence of local moves presented in Fig.\;\ref{MJ_102} (and an isotopy in $\mathbb{R}^3$). The moves $M_1, M_2, M_3, M_4$ are called \emph{cup move, cap move, band-slide, band-pass} respectively.

\begin{figure}[ht]
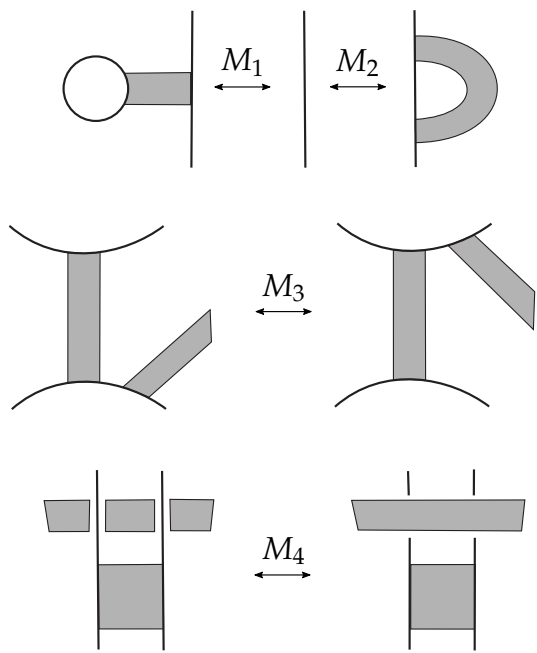

\begin{center}
\begin{lpic}[b(0.5cm)]{./pictures/MJ_102(7cm)}
   \lbl[b]{57,140;$M_1$}
   \lbl[b]{85,140;$M_2$}
   \lbl[b]{67,85;$M_3$}
   \lbl[b]{67,21;$M_4$}
	\end{lpic}
	\caption{Allowed moves on a link with bands.\label{MJ_102}}
	\end{center}
\end{figure}

\begin{proof}[Proof of Thm.\;\ref{thm2}.]

It follows directly from the correspondence between markers and bands (see Fig.\;\ref{MJ_100}), correspondence between moves $M_1, M_2, M_3, M_4$ and $\Omega_6', \Omega_6, \Omega_7, \Omega_8$ respectively and the independence of any move from the set $\{\Omega_6, \Omega_6', \Omega_7, \Omega_8\}$ from the other nine types of Yoshikawa moves.

\end{proof}


\begin{thebibliography}{99}

\bibitem{Fox62} R.H. Fox, A quick trip through knot theory, in \emph{Topology of 3-manifolds and related topics} (Prentice-Hall, 1962), 120-167.
\bibitem{Jab16} M. Jab\l onowski, \emph{On a banded link presentation of knotted surfaces}, J. Knot Theory Ramifications {\bf 25} (2016), 1640004. (11 pages).
\bibitem{JKL13} Y. Joung, J. Kim and S. Y. Lee, Ideal coset invariants for surface-links in $\mathbb{R}^4$, \emph{J. Knot Theory Ramifications} {\bf 22} (2013), 1350052. (25 pages).
\bibitem{JKL15} Y. Joung, J. Kim and S. Y. Lee, On generating sets of Yoshikawa moves for marked graph diagrams of surface-links, \emph{J. Knot Theory Ramifications} {\bf 24} (2015), 1550018. (21 pages)
\bibitem{Kam89} S. Kamada, Non-orientable surfaces in 4-space, \emph{Osaka J. Math.} {\bf 26} (1989), 367-385.
\bibitem{KSS82} A. Kawauchi, T. Shibuya, and S. Suzuki, Descriptions on surfaces in four-space, I; Normal forms, \emph{Math. Sem. Notes Kobe Univ.} {\bf 10} (1982), 72-125.
\bibitem{KeaKur08} C. Kearton, V. Kurlin, All 2-dimensional links in 4-space live inside a universal 3-dimensional polyhedron, \emph{Algebraic and Geometric Topology} {\bf 8}(3) (2008), 1223-1247.
\bibitem{Lom81} S.J. Lomonaco, Jr., The homotopy groups of knots I. How to compute the algebraic 2-type, \emph{Pacific J. Math.} {\bf 95} (1981), 349-390.
\bibitem{Swe01} F.J. Swenton, On a calculus for 2-knots and surfaces in 4-space, \emph{J. Knot Theory Ramifications} {\bf 10} (2001), 1133-1141.
\bibitem{Yos94} K. Yoshikawa, An enumeration of surfaces in four-space, \emph{Osaka J. Math.} {\bf 31} (1994), 497-522.

\end{thebibliography}
\end{document}